\documentclass[12pt,twoside]{amsart}
\usepackage{bbm}
\usepackage{enumerate}
\usepackage[all]{xy}
\usepackage[hyperindex=true,plainpages=false,colorlinks=false,pdfpagelabels]{hyperref}
\usepackage{verbatim}

\theoremstyle{plain}
\newtheorem{The}{Theorem}
\newtheorem*{The*}{Theorem}
\newtheorem{Pro}{Proposition}
\newtheorem{Lem}{Lemma}
\newtheorem{Cor}{Corollary}
\newtheorem*{Cor*}{Corollary}

\theoremstyle{definition}
\newtheorem*{Def}{Definition}
\newtheorem{Rem}{Remark}

\newtheorem*{Rem*}{Remark}

\numberwithin{equation}{section}

\newcommand{\dvector}[1]{{\left(\begin{matrix}#1\end{matrix}\right)}}

\DeclareMathOperator{\del}{\partial}

\newcommand{\R}{\mathbb{R}}

\newcommand{\C}{\mathbb{C}}

\renewcommand{\H}{\mathbb{H}}  

\newcommand{\CP}{\mathbb{CP}}

\newcommand{\ii}{\mathbbm i}
\newcommand{\jj}{\mathbbm j}

\begin{document}

\title{Constrained Willmore and CMC Tori in the $3-$sphere}

\author{Lynn Heller}

\address{ Institut f\"ur Mathematik\\  Universit{\"a}t T\"ubingen\\ Auf der Morgenstelle
10\\ 72076 T\"ubingen\\ Germany
 }

\email{lynn-jing.heller@uni-tuebingen.de}


\date{\today}

\thanks{The author is supported by The Sonderforschungsbereich Transregio 71}

\begin{abstract} 
Constrained Willmore surfaces are critical points of the Willmore functional under 
conformal variations. As shown in \cite{B} one can associate to any 
conformally immersed constrained Willmore torus $f$ a compact Riemann surface $
\Sigma,$ such that $f$ can be reconstructed in terms of algebraic data on $\Sigma$. 
Particularly interesting examples of constrained Willmore tori are the tori with constant 
mean curvature (CMC) in a $3-$dimensional space form. It is shown in \cite{Hi} and in 
\cite{PS} that the spectral curves of these tori are hyperelliptic.
In this paper we show under mild conditions that a constrained Willmore torus $f$ in 
$S^3$ is a CMC torus in a $3-$dimensional space form if its spectral curve has the structure of a CMC spectral curve.
 \end{abstract}
\maketitle


\section{Introduction}\label{sec:intro}
The study of constant mean curvature (CMC) tori is classical in differential geometry. Hopf conjectured that the only compact CMC surface in $\R^3$ would be the round sphere. This was proven by Alexandrov \cite{A} in the case of embeddings in 1956. The first counterexample  was found by Wente \cite{W} in 1986 in the class of immersed tori. In fact there exist a great variety of immersed CMC tori in $3-$dimensional space forms. Hitchin \cite{Hi} and Pinkall and Sterling \cite{PS}
found independently that CMC tori in $3-$space form an integrable system, i.e., one can associate to every CMC torus a Riemann surface $\Sigma$ - the spectral curve - and the immersion can be reconstructed in terms of algebraic data on $\Sigma.$ For CMC tori $ \Sigma$ is hyperelliptic. Based on these results Bobenko \cite{Bob1} constructed all CMC tori in $3-$dimensional space forms explicitly  in terms of theta functions.

On the other hand, constrained Willmore surfaces are critical points of the Willmore functional $\int (H^2 +1)dA$  under conformal variations, see \cite{BoPetP} and \cite{KS}. Compact constrained Willmore minimizers can be viewed as optimal  realizations of the underlying Riemann surface in $3-$space. The Willmore functional is invariant under conformal transformations of the ambient space, thus we do not distinguish between M\"obius equivalent surfaces. 
Examples of constrained Willmore tori are CMC tori in a space form. Another related class of constrained Willmore tori were found by Babich and Bobenko \cite{BaBob}. These tori consists of two non compact components, each of them has constant mean curvature in $H^3$. The components are glued together over the infinity boundary of the hyperbolic space.
Its is shown in \cite{B} that  constrained Willmore tori in $S^3$ also form an integrable system. The spectral curve $\Sigma$ is here either a $4-$fold covering or a double covering of $\C P^1.$ 

Two natural questions are whether both definitions of a spectral curve coincide in the case of 
CMC tori and whether the structure of the spectral curve already determines the 
immersion type. The first question is answered in the affirmative by \cite{B}.
 The constrained Willmore spectral curve of a CMC torus corresponds to the harmonic map spectral curve of its  Gau\ss  $ $ map. The second assertion is not true in general. Counter examples are given by a simple factor dressing of certain homogenous tori, viewed in the constrained Willmore setup. Simple factor dressing preserves  
the spectral curve but the resulting torus is in general not CMC.
We rule out this case  by demanding the
uniqueness of the reconstruction of the surface up to M\"obius transformations from the spectral data. 

In this paper we first recall the definition of the spectral curve $\Sigma$ for CMC tori due to\cite{Bob2} and \cite{Hi}. 
This spectral curve  is hyperelliptic and its genus $g$ is the geometric genus of the immersion. For the reconstruction it is more convenient to consider a (possibly singular) curve $\tilde \Sigma$ for which $\Sigma$ is 
the normalization. The  genus $p$  of $\tilde \Sigma$ is the arithmetic genus of the immersion, see \cite{Hi}. The reconstruction of the immersion is then given by the Sym-Bobenko formula \cite {Bob2}. 

Then we briefly discuss the definition of the spectral curve for general constrained Willmore tori introduced in \cite{B}. The so defined spectral curve is shown in \cite{B} to coincide with the one defined in \cite{BLPP} for conformally immersed tori and is by definition a smooth curve. Thus we can use Theorem 4.2 of \cite{BLPP} for the reconstruction of a constrained Willmore immersion. We call a immersion simple if this reconstruction is unique up to M\"obius  transformations.

For CMC tori simple means that the arithmetic spectral genus equals the geometric spectral genus of the immersion. The main theorem of the paper is the following. 
\begin{The}
A simple constrained Willmore torus in $S^3$ is a CMC torus in a space form if and only if its spectral curve has the structure of a CMC spectral curve.
\end{The}

A necessary condition is that a certain anti-holomorphic involution on $\Sigma$ has fixpoints. This excludes the tori constructed by Babich and Bobenko  \cite{BaBob} mentioned before and is always satisfied if the spectral genus is even. In the case of spectral genus $g \leq 2$ the following holds: It is well known that simple tori of spectral genus $0$ are always homogenous. We show that a simple immersion $f: T^2 \rightarrow S^3$ of spectral genus $1$ is always equivariant, i.e., it  has a $1-$parameter group of M\"obius symmetries. Thus it is always CMC in a space form by \cite{H1}. If the spectral genus is $2$ we show that simple constrained Willmore tori are either equivariant  or CMC in a space form.

\section{Spectral curve for CMC tori} \label{CMCteil}

This definition of the spectral curve is due to \cite{Hi}  and \cite{Bob2}. We define it only for the case of CMC tori in $S^3.$
Let $f : T^2 \rightarrow S^3$ be a conformally immersed torus with mean curvature $H$ and  
let $\alpha =  f^{-1} df$ be the corresponding Maurer-Cartan form.  We denote by $W$ the pull back of the spinor bundle of $S^3.$ This is a complex rank $2$ vector bundle with a quaternionic structure $j$ and a symplectic form $\hat \omega.$  

Let  $\alpha =\alpha' + \alpha''$ be the splitting of $\alpha$ into its complex linear and complex anti-linear part. 
Then we can associate to $f$ a family of connections
\begin{equation}
\nabla^{\lambda} = \nabla + \tfrac{1}{2}(1+ \lambda^{-1} )(1 + i H) \alpha' + \tfrac{1}{2}(1+ \lambda)(1 - iH)\alpha''), \quad \lambda \in \C_*,
\end{equation}
where $\nabla$ is the trivial connection on $W.$
This family of connections is flat if and only if $H \equiv const$ in $S^3$ and  has the symmetry 
$$\nabla^{\bar \lambda^{-1}} = j^{-1} \nabla^{\lambda} j.$$
Thus for $\lambda \in S^1$ the connection $\nabla^{\lambda}$ is unitary. The following theorem shows how to reconstruct CMC immersions in a $3-$dimensional space forms from this associated family of flat connections. Moreover, it shows that we can restrict ourselves without loss of generality to the associated family for minimal surfaces in $S^3$ given by 
\begin{equation}\label{flat}
\nabla^{\lambda} = \nabla + \tfrac{1}{2}(1+ \lambda^{-1} ) \alpha' + \tfrac{1}{2}(1+ \lambda)\alpha''), \quad \lambda \in \C_*,
\end{equation}
where $\nabla$ is a unitary connection on $W.$
\begin{The}[\cite{Bob2}]\label{sym}
Let $\nabla^{\lambda}$ be a family of flat connections on $W$ given in (\ref{flat}) and $\lambda_0, \lambda_1 \in \C$. Further let $X_\lambda,$ det$X_\lambda = 1,$
be a parallel frame of $\nabla^{\lambda}.$   

For two distinct $\lambda_0, \lambda_1 \in S^1$ the map $f$ given by 
\begin{equation}
f = X_{\lambda_0}^{-1}X_{\lambda_1} 
\end{equation}
is well defined on $\C$ and has constant mean curvature $H = i \frac{\lambda_0 + \lambda_1}{\lambda_0 - \lambda_1}$ in $S^3$ 

If $\lambda_0 = \lambda_1 \in S^1$ then 
\begin{equation}
f = X_{\lambda_0}^{-1} \frac{\del X}{\del \lambda}|_{\lambda_0}
\end{equation}
is a CMC $1$ surface in $\R^3$.

And for $\lambda_1 = \bar \lambda_0^{-1}$ and $|\lambda_0|^2 < 1,$
$f$ given by 
\begin{equation}
f =  X_{\lambda_0}^{-1}X_{\lambda_1} 
\end{equation}

is CMC in $H^3$ with mean curvature $H= \frac{1+ |\lambda_0|^2}{1- |\lambda_0|^2}$.\\
Further, all CMC tori in $3-$space can be obtained this way.
\end{The}
\begin{Rem}
The formulas above are called the Sym-Bobenko formulas and the $\lambda_i$ are called Sym points. By choosing $\lambda_0 = - \lambda_1 \in S^1$ we obtain a minimal immersion in $S^3$. 
Thus in order to obtain all CMC tori in $3-$dimensional space forms it is sufficient to write down all possible $1-$parameter families of flat connections of the form (\ref{flat}). The first step is to define a spectral curve.\\
\end{Rem}

Let $T^2= \C /\Gamma$ and let $H^{\lambda}_x(\gamma)$ denote the holonomy of the connection $\nabla^{\lambda}$ along $\gamma \in \Gamma$ with respect to the base point $x\in T^2.$ We choose $\gamma$ to be one of the generators of $\Gamma.$  By \cite{Hi} the holonomy $H^{\lambda}_x(\gamma)$ is generically diagonalizable and has distinct eigenvalues. The spectral curve $\Sigma$ is now the normalization and compactification of the analytic variety 
$$\{(\eta, \lambda) \in \C_* \times \C_* \ | \ f(\eta, \lambda) = 0\} \quad \quad \text{with} \quad \quad f(\eta, \lambda) = \text{det}(H_x^{\lambda}(\gamma) -  \eta Id).$$ 
Since $H_p^{\lambda}(\gamma)$ is  $2\times 2,$  $\Sigma$ is hyperelliptic and  we have a hyperelliptic involution $\sigma$ on $\Sigma.$ Further, an anti-holomorphic involution $\rho$ is given by the quaternionic structure $j.$
By \cite{Hi}
the spectral curve of a minimal immersion is branched over $\lambda = 0$ and $\lambda = \infty$. Therefore we obtain
$$\Sigma: \eta^2 = \lambda \Pi_{i = 1}^g \bar q_i^{-1} (\lambda - q_i)(\lambda - \bar q_1^{-1}),$$
where $q_i \neq  0, \infty,$ and $\bar q_i^{-1}$ are the odd order roots of the function
$f(\eta, \lambda)$ without multiplicity. The change of the base point $x \in T^2$ corresponds to a conjugation of the holonomy matrix. Thus the eigenvalues and hence the spectral curve is independent of $x\in T^2.$ Further, since $\Gamma$ is a lattice,  $\Sigma$ is also independent of the choice of $\gamma \in \Gamma$.  \\

Next we want to define a $T^2-$family of complex holomorphic line bundles over the spectral curve and collect all properties these line bundles have. 
For a fixed $x \in T^2$ let $\mathcal L_x$ denote the line bundle over $\Sigma$ given by the eigenspace of the holonomy. To be more explicit, at a  generic point  $(\eta,\lambda) \in \Sigma,$ where $H^{\lambda}_x(\gamma)$ has $2$ distinct eigenvalues,  we define the   fiber $\mathcal L_x|_{(\eta,\lambda)}$  to be the $1-$dimensional eigenspace of $H^{\lambda}_x(\gamma)$ with respect to the eigenvalue $\eta.$ The bundle is well-defined and extends holomorphically  
to a line bundle over all $\Sigma.$
\begin{Pro}[\cite{Hi}, section 7]
Let $x_0 \in T^2.$ Then map 
$$\Psi: T^2 \rightarrow  Jac(\Sigma), x \mapsto \mathcal L_x L_{x_0}^{-1} $$ is a group homomorphism. 
\end{Pro}
 In order to make the map $\Psi$ explicit, it is necessary to compute the degree of $\mathcal L_x.$ Let $\sigma$ denote the hyperelliptic involution on $\Sigma.$
Then the bundles $\mathcal L_x$ and $\sigma^*\mathcal L_x$ are subbundles of the trivial bundle  $\Sigma \times W_x.$ Since $W$ is the pull-back of the spin bundle of $S^3$, there exists a symplectic form $\hat \omega$ on $W$. Therefore the evaluation of the symplectic form $\hat\omega_x$ on $\mathcal L_x \otimes \sigma^*\mathcal L_x$ defines a holomorphic map to $\C$. Thus  $\hat\omega_x$  is a section in the line bundle $\mathcal L^*_x \otimes (\sigma^*\mathcal L^*)_x$ over 
$\C P^1$. It vanishes exactly at those points, where $\mathcal L_x$ and $\sigma^*\mathcal L_x$ coincides. Therefore the zeros of $\hat \omega_x$ do not depend on $x.$ Obviously $\hat \omega_x$ vanishes at branch points of $\Sigma,$ thus  the degree of $\mathcal L^{*}_x$ is at least $g+1$.  

\begin{Def}
Let $f: T^2 \rightarrow S^3$ be a conformal immersion of constant mean curvature and $\Sigma$ its spectral curve. The genus $g$ of $\Sigma$ is called the geometric spectral genus of $f$. 
The arithmetic spectral genus $p$ of $f$ is the genus of the possibly singular curve 
$$\tilde \Sigma: \eta^2 =\lambda \Pi_{i = 1}^{p}\bar q_i^{-1}(\lambda - q_i)(\lambda - \bar q_i^{-1}),$$

 where the $q_i$'s are the  zeros of the symplectic form $\hat \omega$ counted with multiplicity.
\end{Def}

 \begin{Pro}[\cite{Hi}, Theorems 8.1 and  7.15]
The line bundle $\mathcal L^*_x,$ $x \in T^2,$ has degree $p+1,$ where $p$ is the arithmetic spectral genus of $f$. Further, the bundle $\mathcal L_x$ is non special for all $x\in T^2,$ i.e., $K \mathcal L_x^*$ has no holomorphic sections. 
\end{Pro}

For a given spectral curve and a family of line bundles $\Psi(T^2)$ compatible with all involutions, we can reconstruct by \cite{Hi} an associated family of flat connections of the form (\ref{flat}) and thus also an associated family of constant mean curvature tori.
\begin{The}[\cite{Hi}, Theorem 8.1]\label{hitchin}
Let $\tilde \Sigma$ be a (possibly singular) hyperelliptic curve over $\C P^1$ defined by the equation
$$\eta^2 =\lambda \Pi_{i = 1}^{p}\bar q_i^{-1}(\lambda - q_i)(\lambda - \bar q_i^{-1}),$$
And let $$\Psi: T^2 \rightarrow Pic_{p+1}(\Sigma), x \mapsto \mathcal L_x$$ be a group homomorphism. 
Suppose all $\mathcal L_x \in \Psi(T^2)$  are non special and $ \rho^* \mathcal L_x = \mathcal L_x \jj.$
Further, suppose $\mathcal L^*_x \otimes \sigma ^*\mathcal L^*_x = L(2p+2)$,
where $L(2p+2)$ is the bundle of degree $2p+2$ on $\CP^1.$ \\

Then we can construct a family of flat connections $\nabla^{\lambda}$ on $W$ for which the lines $\mathcal L|_{(\eta, \lambda)}$ and $\mathcal L|_{(-\eta, \lambda)}$ define a parallel frame of $W$ w.r.t. $\nabla^{\lambda}.$ 
This family has the form
 $$\nabla^{\lambda} =  \nabla + \tfrac{1}{2}(1+ \lambda^{-1} ) \alpha' + \tfrac{1}{2}(1+ \lambda)\alpha'',$$
 for some unitary connection $\nabla$.
\end{The} 

\section
{The Constrained Willmore Case}
In order to define the associated family of flat connections for constrained Willmore tori, we recall come facts on quaternionic surface theory developed in \cite{BuFLPP}. 
The conformal geometry of the $4-$sphere is modeled by the projective geometry of the quaternionic projective line $\H P^1.$  
A map $f: T^2 \rightarrow S^4 \cong \H P^1$ is given by a quaternionic line subbundle $L := f^*\mathcal T$ of the trivial $\H^2-$bundle $V := T^2 \times \H^2, $ where $\mathcal T$ is the tautological bundle of $\H P^1$. Using affine coordinates we can consider $$S^3 \subset \H \hookrightarrow \H P^1,  \lambda \mapsto [\lambda, 1].$$ 
An oriented round $2-$sphere in $S^4$ is given by a linear map $S: \H^2 \rightarrow \H^2$ with $S^2 = -1,$ i.e., a complex structure on $\H^2$.  
For given $f$  there exist a natural map $S$ from the torus into the space of oriented $2-$spheres - the conformal Gau\ss $ $	map. It assigns to every point $p$ of the surface a $2-$sphere $S_p$ through the point $f(p)$ with the same tangent plane and the same mean curvature. The pair $(V, S)$ is then a complex quaternionic vector bundle.The  trivial connection $\hat \nabla$ on $V$ splits into a complex linear part $\hat\nabla'$ and a complex anti-linear part $\hat\nabla''$ with respect to $S.$  
 We denote by $A$ and $Q$ the $S-$anti-commuting parts of $\hat\nabla'$ and $\hat\nabla''$ respectively.
 
\begin{The}[\cite{B}]
Let $f: T^2 \rightarrow S^3$ be a conformal immersion. It is constrained Willmore if there exist a $1-$form $\nu \in \Omega^1(\mathcal R)$ such that 
$$d^{\hat\nabla} (2 *A  + \nu) = 0,$$
where $\mathcal R = \{B \in \text{End}(V) \ |\  \text{Im}B \subset L \subset \text{Ker}B\}. $ 
\end{The}
The $1-$form $\nu$ is the Lagrange multiplier. It is uniquely determined unless the immersion is isothermic. Let $A_0$ be the $1-$form defined by $2*A_0 = 2 *A  + \nu.$ Consider $V$ as a complex rank $4$ bundle with complex structure $i$ given by the left multiplication of the quaternionic $\ii$ together with an anti-linear endomorphism $j$ which is the right multiplication by the quaternion $\jj.$
The Euler-Lagrange equation for constrained Willmore tori is equivalent to the flatness of the associated family of $SL(4, \C)-$connections
$$\nabla^{\mu} = \hat\nabla +  (\mu -1)\frac{1- i S}{2}A_0 + (\mu^{-1} -1)\frac{1+ i S}{2}A_0,$$
for $\mu \in \C_*.$
This associated family $\nabla^{\mu}$ has the symmetry
$$\nabla^{\bar \mu^{-1}} = j^{-1} \nabla^{\mu}j.$$
Thus for all $\mu \in S^1 \subset \C_*$ the connection $\nabla^{\mu}$ is quaternionic.\\

For a fixed point $x\in T^2 = \C /\Gamma$ consider the holonomy representations $H_x^\mu$ of the associated family $\nabla^{\mu}.$ The representations are holomorphic in $\mu$ and the holonomy $H_x^\mu$ to different basis points are conjugated. Thus the eigenvalues are independent of $x\in T^2.$
Since the first fundamental group of $T^2$ is abelian, we get that every simple eigenspace of  $H_x^\mu(\gamma_0)$ for $\gamma_0 \in \Gamma$ is a simultaneous eigenspace for all $\gamma \in \Gamma.$ 

\begin{The}[\cite{B}, Proposition 3.1 and  Theorem 5.1]
Let $f: T^2 \rightarrow S^3$ be a constrained Willmore torus. Then the holonomy representation of the associated family $\nabla^{\mu}$ belongs to the following $2$ cases:
\begin{enumerate}
\item For $\gamma \in \Gamma$  the holonomy $H_x^\mu(\gamma)$ has  $4$ distinct eigenvalues for generic $\mu \in \C_*$. These eigenvalues are non-constant in $\mu.$  
\item All holonomies $H_x^\mu(\gamma)$ have a $2-$dimensional common eigenspace with eigenvalue $1$ and $H_p^\mu(\gamma)$ has  $2$ distinct and non constant eigenvalues for generic $\mu \in \C_*$. 
\end{enumerate}
\end{The}

The spectral curve $\Sigma$ of a constrained Willmore torus is the normalization and compactification of the $1-$dimensional analytic variety 
$$\{(\eta, \mu) \in \C_* \times \C_*| f(\eta, \mu) = 0\} \quad \quad \text{with} \quad \quad f(\eta, \mu) = \text{det}(H_p^{\mu}(\gamma) -  \eta Id).$$

This spectral curve differs from the one defined in \cite{BLPP}, $\hat\Sigma,$ only by two points which we denote by $0$ and $\infty$. These points corresponds to the ends of $\hat \Sigma.$ 
It is a $4-$ fold covering of $\C P^1,$ if the holonomy representation of $\nabla^{\mu}$ belongs to case $(1)$ and it is hyperelliptic if the holonomy representation of $\nabla^{\mu}$ belongs to case $(2)$. We restrict ourselves in the following to the case, where the spectral curve is a $4-$fold covering of $\C P^1$ since the second case only occurs for CMC surfaces by \cite{B}, Proposition (6.12).\\

We have two involutions $\rho$ and $\sigma$ on the spectral curve. Because of the symmetry $j$ of $\nabla^{\mu}$the involution $\rho$ covers the involution $\mu \mapsto \bar \mu^{-1}$ on $\C P^1.$ 
By the following lemma the involution $\sigma$ fixes the spectral parameter $\mu$.

\begin{Lem}\label{symplectic}
Let $f: T^2 \rightarrow S^3$ be a constrained Willmore torus and $\nabla^{\mu}$ the corresponding associated family of flat connections. Further let $H_x^{\mu}(\gamma)$ denote the holonomy of $\nabla^{\mu}$ at a base point $x \in T^2$ along $\gamma \in \Gamma.$
If $\eta$ is an eigenvalue of $H_x^{\mu}(\gamma)$
 then  $\eta^{-1}$ is also an eigenvalue of  $H_x^{\mu}(\gamma)$.
 \end{Lem}
 \begin{proof}
 In the quaternionic setup the conformal $3-$sphere in $\H P^1$ is  given by the null lines of the indefinite product $(.,.)$ on $\H^2$
$$(v,w) = \bar v_1 w_2 +\bar v_2 w_1.$$
The  $\ii-$anti-commuting part of $(.,.)$ gives rise to a symplectic form $\omega$ on $\H^2.$ 
 Consider the bundle $L^\perp \subset V^*,$ where $V^*$ is a complex quaternionic bundle with respect to the complex structure $-i$. We can define a  family of  connections $(\nabla^{\perp})^{\mu},$ for $\mu \in \C_*$ on $L^{\perp}$ by 
 $$(\nabla^{\perp})^{\mu} = \nabla^{\perp} +  (\mu -1)\frac{1 + i S^\perp}{2}A_0^\perp + (\mu^{-1} -1)\frac{1- i S^\perp}{2}A_0^\perp.$$
This family is dual to the family of connections 
$$\tilde \nabla^{\mu} = \nabla +  (\mu -1)Q_0\frac{1-i S}{2} + (\mu^{-1} -1)Q_0\frac{1+ i S}{2},$$
with respect to $(.,.)$,
since $A^\perp = - Q^*$, $Q^\perp = -A^*$ and $\nu^\perp = - \nu^*.$

Both families of connections are flat if and only if the immersion is constrained Willmore. By duality of these families we have:
If $\eta$ is an eigenvalue of the holonomy of $(\nabla^\perp)^\mu$ then $\eta^{-1}$ is an eigenvalue of the holonomy of $\tilde \nabla^{\mu}.$ 
On other hand, we have for conformal immersions into $S^3$ that $L = L^\perp$ and $S = S^* = S^{\perp}$ and thus $A = - Q^*$. Therefore $(\nabla^{\perp})^\mu = \nabla^\mu.$ Further, $\nabla^{\mu}$ is gauge equivalent to $\tilde \nabla^{\mu}$. The gauge is given by 
$$\mathbbm g =\tfrac{1}{2} ((\mu + 1) - i(\mu -1)S).$$
This proves the lemma.
\end{proof}

For a fixed $x \in T^2$ consider the symplectic form $\omega$ on $V_x$ defined in Lemma \ref{symplectic}.
At a generic point $(\eta, \mu) \in \Sigma$ there exist  a basis of $\C^4$ such that  $H_x^\mu(\gamma)$ is given by 
$$H_x^\mu(\gamma) = \dvector{\eta & 0 & 0 & 0\\ 0 &\eta^{-1}&0&0\\ 0 &0&\tilde \eta& 0\\ 0 &0&0&\tilde \eta^{-1}}.$$
Let $\tilde {\mathcal L}_x \subset \Sigma \times \C^4$ be the line bundle which at generic points  $(\eta, \mu) \in \Sigma$ coincides with the eigenspace of $H^\mu_x(\gamma)$ to  the eigenvalue $\eta.$ Then the line bundle $\mathbbm g(\sigma^*\tilde {\mathcal L}_x)$ corresponds generically  to the eigenspace of the holonomy of $\tilde \nabla^{\mu}$ w.r.t. the eigenvalue $\eta^{-1}.$ The evaluation of $\omega$ on $\tilde {\mathcal L}_x \otimes \mathbbm g(\sigma^*\tilde {\mathcal L}_x)|_{(\eta, \mu)}$ is non-degenerate at generic points, where we have $4$ distinct eigenvalues, and thus it defines a holomorphic map 
$$\tilde {\mathcal L}_x \otimes \mathbbm g(\sigma^*\tilde {\mathcal L}_x )\rightarrow \C.$$ This map vanishes, if and only if the lines $\tilde {\mathcal L}_x $ and $\mathbbm g(\sigma^* \tilde {\mathcal L}_x) $ coalesce at $(\eta, \mu)$. This is  independent on $x$. The bundles $\mathbbm g(\sigma^* \tilde {\mathcal L}_x) $ and $\sigma^* \tilde {\mathcal L}_x$  are holomorphic isomorphic. Therefore we obtain
 \begin{Lem}\label{omega}
 The bundle $\tilde {\mathcal L}_x \otimes \sigma^*\tilde {\mathcal L}_x$ is independent of $x \in T^2$ as a complex holomorphic line  bundle.
 \end{Lem}

 \subsection{The reconstruction of a constrained Willmore immersion}
 Let $\Sigma$ be a connected compact Riemann surface with an anti-holomorphic involution $\rho.$ Further fix a real subtorus $Z = \Psi(T^2)$ of dimension $0,1$ or $2$ of the Jacobian of $\Sigma.$ For $x \in T^2$ the line bundle $\tilde {\mathcal L}_x$ over $\Sigma$ is by construction a complex holomorphic line subbundle of $\Sigma \times \C^4.$ Thus it defines a map from $\Sigma$ to $\C P ^3.$ A  quaternionic structure on $\C^4$ is a real linear endomorphism $\jj$ with $\jj^2 = -1$ anti-commuting with $i$.  By fixing such a quaternionic structure $\jj$ on $\Sigma \times \C^4$ we obtain a canonical  isomorphism between $\C^4$ and $\H^2.$ This isomorphism induces a map $\pi_{\H}$ between  $\CP^3$ and $\H P^1$ which is called twistor projection. In our case the quaternionic structure is given by $\rho.$ The main theorem for the reconstruction of a conformally immersed torus is the following.
\begin{The}[\cite{BLPP},Theorem 4.2]\label{Immersionreconstruction}
Let $f: T^2 \rightarrow S^3$ be a conformal immersion whose spectral curve $\Sigma$ has finite genus. Then there exist a map
\begin{equation}
 F : T^2 \times \Sigma \rightarrow \C P^3, 
\end{equation}
such that 
\begin{itemize}
\item $F(x, - ): \Sigma \rightarrow \C P^3$ is an algebraic curve, for all $x \in T^2.$\\
\item The original conformal immersion $f : T^2 \rightarrow S^3$ is obtained by the twistor projection of the evaluation of $F$ at the points at infinity:
$$f = \pi_{\H} F(-, 0) = \pi_{\H} F(-, \infty).$$
\\
\end{itemize}
\end{The}

For given $\Sigma$ with marked points $0$ and $\infty$ and a $T^2-$family of holomorphic line bundles $Z$ in Jac$(\Sigma)$ the map $F$ is in general not unique. In other words, the  immersion $f$ is in general not uniquely determined by the  spectral curve  and $Z$. \\

Let $E$ be a complex holomorphic line bundle over a Riemann surface $M.$ By the Kodaira embedding theorem there exist a holomorphic map $s$ from $E$ to $\C P^n$ if and only if the space of holomorphic sections of the line bundle $E^*$ is at least complex $(n+1)-$dimensional, i.e., deg $(E^*) \geq n+1$. The space of holomorphic sections of $E^*$ is exactly $(n+1) -$dimensional if and only if the map $s$ is unique up to a $PSL(n+1,\C)$ action on $\CP^n$. In our case  we have $n = 4$ and the line bundles in $Z$ must be compatible with the quaternionic structure $\jj.$ Elements of $PSL(4, \C)$ compatible with $\jj$ act on $\H P^1$ as M\"obius transformations. Thus  $f = \pi_{\H}(F(-, 0))$ is uniquely determined up to M\"obius transformations of $S^4$ if and only if the space of holomorphic sections of $\tilde {\mathcal L}_x$ is $4-$dimensional. \\

\begin{Def}
A conformal immersion is simple if it is uniquely determined by $\Sigma$ and $Z$ up to M\"obius transformations of $S^4.$  \end{Def}

Instead of the quaternionic line bundle $L$, we can also consider the quaternionic line bundle $V/L$. Let $S$ denote the conformal Gau\ss $ $ map. 
The property $f(p) \in S_p$ translated into the quaternionic language is   $S_p L_p \subset L_p.$ 
Thus the projection of $S$ to $V/L$ defines a complex structure and the projection of $\hat \nabla''$ defines a holomorphic structure on $V/L$. By definition the projection of a constant sections of $V$ to $V/L$ is holomorphic.  The map $f: T^2 \rightarrow S^3$ is then given by the quotient of these sections, since
$$\pi_{V/L}(0,1) + \pi_{V/L}(1,0) f = 0_{V/L}.$$ The next Lemma states that the only holomorphic sections of $V/L$ comes from the constant sections of $V,$ if $f$ is simple. In this case the immersion is uniquely determined by its quotient bundle up to M\"obius transformations of $S^4.$

\begin{Lem}\label{h0vl}
Let $f: T^2 \rightarrow S^3$ be a simple conformal immersion and $V/L$ be the associated quotient bundle, then  the space $H^0(V/L)$ is quaternionic $2-$dimensional.
 \end{Lem} 
\begin{proof}

If $H^0(V/L)> 2$ then there exist at least $3$ quaternionic linear independent holomorphic sections. The quotient of any two of them yield a map from $T^2$  to $\H P ^1.$ Thus in this case we get at least two maps $f$ and $\tilde f$ which are not M\"obius equivalent. The corresponding quotient  bundles $V/L$ and $V/\tilde L$ are holomorphic isomorphic. Thus $f$ and $\tilde f$ have the same spectral curve and $Z$ and the map $F$ cannot be unique.  
\end{proof}

\begin{The}\label{CMCs3}
Let $f: T^2 \rightarrow S^3$ be a simple constrained Willmore immersion with spectral curve $\Sigma$. On $\Sigma$ we have a fix point free anti-holomorphic involution \linebreak $\rho : (\eta, \mu) \mapsto (\bar \eta, \bar\mu^{-1})$ and an involution $\sigma : (\eta, \mu) \mapsto (\eta^{-1}, \mu)$. If the quotient $\Sigma/\sigma \cong \C P^1$ and $\rho \circ \sigma$ has fixpoints, then $f$ is a CMC torus in a space form.
\end{The}
\begin{proof}
Let $f: T^2 \rightarrow S^3$ be a simple,  conformal and constrained Willmore immersion. Let $\nabla^{\mu}$ be  its associated family of complex flat connections of constrained Willmore surfaces on the bundle $V = T^2 \times \H^2 \cong T^2 \times \C^4.$ 
The spectral curve  $\Sigma$ is a $4-$fold covering of $\C P^1$  and has two marked points $0$ and $\infty$ corresponding to the ends of the spectral curve $ \hat \Sigma$. These points are by \cite{B} fixed under the involution $\sigma$ and interchanged by the involution $\rho.$ By assumption the quotient $\Sigma/\sigma$ is also $\C P^1.$ Thus the spectral curve $\Sigma$ is a hyperelliptic curve and $\sigma$ is the hyperelliptic involution.  Let $\lambda$ be a holomorphic coordinate of $\Sigma/\sigma$ such that $0 \in \Sigma$ is the point over $\lambda  =0$ and $\infty \in \Sigma$ is the point over $\lambda = \infty$. Since the involutions $\rho$ and $\sigma$ commute and $\rho$ interchanges the points $0$ and $\infty$ on $\Sigma$, $\rho$ also interchanges the points $\lambda = 0$ and $\lambda = \infty$ on $\Sigma/\sigma.$ An anti-holomorphic involution on $\C P^1$  interchanging $\lambda = 0$ and $\lambda = \infty$ is either the map $\lambda \mapsto \bar \lambda^{-1}$ or the map $\lambda \mapsto - \bar \lambda^{-1}.$ Since $\rho \circ \sigma$ has fixed points, $\rho$ induces the involution $\lambda \mapsto \bar \lambda^{-1}$ on $\Sigma/\sigma$, which fixes the points over $\lambda \in S^1.$ 

By assumption there exist a $T^2-$family  of complex holomorphic line subbundles $\tilde {\mathcal L}_x$ of $\tilde V = \Sigma \times \C^4.$ Therefore we obtain byâ Theorem \ref{Immersionreconstruction} there is a map
$$F:  T^2 \times \Sigma \rightarrow \C P^3, $$
such that  the line bundle $L$ corresponding to the immersion $f$ can be reconstructed by $L = \pi_{\H}F(-, \infty).$
 Thus we can define the quotient bundle $V/L$ and the projection of $\tilde {\mathcal L}_x, $ $x \in T^2,$ to $V/L$ defines a $T^2-$family of complex  line subbundles $\mathcal L_x$ of the topologically trivial complex rank $2$ bundle $V/L.$ 
We want to show that for fixed $x \in T^2$ the  bundle $\mathcal L_x^*$ has degree $g+1.$
By \cite{BLPP} the map $$\Psi: T^2 \rightarrow Jac(\Sigma), x \mapsto \mathcal L_x \mathcal L_{x_0}^{-1}$$ is a group homomorphism for fixed $x_0 \in T^2.$  Thus the degree of $\mathcal L_x$ is constant in $x$. Moreover, since $f$ is simple, the map $F$ is unique up to M\"obius transformations of $S^4.$ Thus the complex holomorphic bundle $\pi_L(F(x,-)) = \mathcal L_x \subset  \Sigma \times(V/L)_x $ is unique up to M\"obius transformations of $\C P^1$ and therefore the space of holomorphic sections of $\mathcal L_x^*$ is  $2-$dimensional by the Kodaira embedding theorem.  

Let $\widetilde{(.,.)}$ denote the parallel Hermitian form on $V/L$ with respect to $\pi_{v/L}(\hat\nabla).$ Then the $\ii-$anti-commuting part of  $\widetilde{(.,.)}$ defines a symplectic structure $\hat \omega$ on $V/L$. The evaluation of $\hat \omega$ on $\mathcal L_x \otimes \sigma^*\mathcal L_x$ is non-degenerate at generic points on $\Sigma$ and vanishes at the branch points of $\sigma.$
Thus the degree $d$ of $\mathcal L^*_x$ satisfies  $(g+1) \leq d$, where $g$ is the genus of $\Sigma.$  By the Riemann-Roch theorem we thus obtain that the line bundle $\mathcal L_x^*$  is  non-special and deg$(\mathcal L_x^*) = g+1.$
Using Theorem \ref{hitchin}, we can therefore define  a family of flat connections $\nabla^\lambda,$ $\lambda \in \C_*$ on $V/L$ of the form
$$ \nabla^{\lambda} = \nabla + \tfrac{1}{2}(1+ \lambda^{-1} ) \alpha' + \tfrac{1}{2}(1+ \lambda)\alpha'',$$
for which  $(\pi_{V/L}(F(-, h)), \pi_{V/L}(F(-, \sigma(h)))$ define a parallel frame of the complex rank two bundle $V/L$. \\

The section $\psi_h = \pi_{V/L}(F(-, h))$ is by definition a holomorphic section with monodromy of $V/L.$ Since $f$  is simple, all holomorphic  sections of $V/L$ with trivial monodromy are given by the projection of constant sections of $V$ by Proposition \ref{h0vl}.   Constant sections of $V$ are parallel sections of $\hat \nabla = \nabla^{\mu = 1}, $ where $\nabla^{\mu}$ is the constrained Willmore associated family of flat connections. This is a complex $4$ dimensional space. Thus there exist at most $2$ trivial connections in the $\nabla^{\lambda}$ family of flat connections. Let $\lambda_0$ and $\lambda_1$ denote the points of $\Sigma/\sigma$ over $\mu = 1$. There are three cases to consider.

In the first case both points $\lambda_0 \neq \lambda_1,$  are fixed under the involution $\rho \circ \sigma,$ then $\lambda_{0}$ and $\lambda_1$  has  length $1.$
Thus the reconstruction by the Sym-Bobenko formula stated in Theorem \ref{sym} gives a CMC immersion in $S^3$.
In the second case $\lambda_0 = \lambda_1$ and we obtain a CMC immersion in $\R^3.$ 
In the third case we have again $\lambda_ 0 \neq \lambda_1,$ but both points are not fixed under $\rho \circ \sigma.$ The reconstruction by the Sym-Bobenko formula we obtain a CMC immersion in $H^3.$ 

It remains to show that the immersion reconstructed by the Sym-Bobenko formula $\tilde f$ is M\"obius equivalent to the original immersion $f.$ This holds since by construction the quotient bundles $V/\tilde L$ and $V/L$ are holomorphic isomorphic. 
\end{proof}

\begin{Cor}\label{simple}
A conformally immersed CMC torus in a $3-$dimensional space form is simple if and only if its arithmetic spectral genus $p$ equals its geometric spectral genus $g$. 
\end{Cor}
\begin{proof}
The degree of the line bundle $\mathcal L_x$ is by definition $p+1.$ If $p>g$ we obtain by Riemann-Roch that  the space of holomorphic sections of $\mathcal L_x^*$ is at least $3-$dimensional. Therefore  the map $F$ cannot be unique by the Kodaira embedding theorem. 
\end{proof}

\begin{Cor}
Let $f$ be a constrained Willmore conformal immersion in $S^3$ with even spectral genus and $\Sigma/\sigma \cong \C P^1,$ then  
the involution $\rho \circ \sigma$ has fixed points. In particular, if $f$ is simple then it is CMC in a space form.
\end{Cor}
\begin{proof}
Let $\Sigma$  be a hyperelliptic  curve with  two marked points $0$ and $\infty$ and the two involutions $\sigma$  and $\rho.$ Suppose $\rho \circ \sigma$ has no fixed points.  Then $\rho$ induces a fixpoint free  involution  on $\C P^1$ interchanging the points $0$ and $\infty$. 
Let $\lambda$ be the holomorphic coordinate on $\C P^1$ such that the point $0\in \Sigma$ lies over $\lambda = 0$ and  the point $\infty \in \Sigma$  lies over $\lambda = \infty.$
Then $\rho$ on $\Sigma/\sigma \cong \C P^1$ is given by
$$\lambda \mapsto -\bar \lambda^{-1}$$
In this coordinate $\Sigma$ is given by
$$\eta^2 = \Pi_{i=1}^{g+1}\bar q_i^{-1}(\lambda - q_i)(\lambda +\bar q_i^{-1}) =: P(\lambda),$$
where $q_i, - \bar q_i^{-1} \in \C$ and $0, \infty$ are the branch points of $\Sigma.$
It is easy to compute that 
$$P(- \bar\lambda^{-1}) = (-1)^{g+1}\bar\lambda^{-(2g + 2)}\overline{P(\lambda)}.$$
Therefore the  map $\rho$  on $\Sigma$ inducing the involution $\lambda \mapsto - \lambda^{-1}$  is given by
$$(\eta ,\lambda)\mapsto (\pm i^{g+1} \bar \eta \bar \lambda^{-(g+1)}, -\bar \lambda^{-1} ).$$
Since $g$ is even this cannot define an involution on $\Sigma$.  \end{proof}

\begin{Pro}\label{g=1}
A simple immersion $f: T^2 \rightarrow S^3$ (not necessarily constrained Willmore) of spectral genus $1$ is equivariant, i.e., it has a $1-$parameter group of M\"obius symmetries. If $\rho \circ \sigma$ has fixed points then it is CMC torus in a space form.
\end{Pro}
\begin{proof}
By assumption the map $F$ is unique. Thus the bundle $L = \pi_{\H}F(-, \infty)$ is unique up to  M\"obius transformations of $S^4.$ If we only consider immersions into a fixed $S^3 \subset S^4,$ we obtain that this reconstruction is unique up to M\"obius transformations of $S^3.$
Since the Jacobian of a torus is the torus itself and the set of line bundles compatible with $\rho$  is only a $S^1,$ the map $\Psi$ has a $1-$dimensional kernel.  
Now let $x$ be the direction in $T^2 = \C /\Gamma$ parametrizing the kernel of $\Psi$. Then the conformal maps $f(x,y)$ and $f(x+ x_0, y)$ have  the same spectral curve $\Sigma$ and $Z.$  Thus there is a M\"obius transformation $M_{x_0}$ of $S^3$ with 
$$f(x + x_0,y) = M_{x_0}f(x,y).$$ 
The map $M: \R \rightarrow$ M\"ob$(S^3)$ is a group homomorphism, thus $f$ is equivariant.\\

The involution $\sigma$ is branched over the two ends of the spectral curve by \cite{B}. Thus by the Riemann-Hurwitz formula we obtain
that there must be two other branch points and $\Sigma/\sigma \cong \C P^1.$ 
\end{proof}
\begin{Rem}
If $\rho \circ \sigma$ has no fixed points, then $f$ is given by the rotation of an elastic figure eight curve in the hyperbolic plane, modeled by the upper half plane, around the $x-$axis. 
\end{Rem}
\begin{Cor}
A simple constrained Willmore immersion in $S^3$ of spectral genus $2$ is either equivariant or CMC in a space form.
\end{Cor}
\begin{proof}
By Riemann-Hurwitz formula the involution $\sigma$ has either $2$ or $6$ branch points. 
If $\sigma$ has $6$ branch points then $\Sigma/\sigma$ is $\C P^1.$ Because the genus is even, the involution $\rho \circ \sigma$ has fixed points. Therefore we get a CMC torus in a space form.\\

If $\sigma$ has $2$ branch points, then $\Sigma/\sigma$ is a torus. Let $x_0, x \in T^2$. We have
$$\mathcal L_x = E_x \otimes \mathcal L_{x_0}.$$
Since the bundle  $\mathcal L^*_x \otimes \sigma^*\mathcal L^*_x$ is independent of $x$ by Lemma \ref{omega}, we obtain
 
$$E_x \otimes \sigma^* E_x = \underline \C.$$
Thus $\mathcal L_x$ lies  in a affine translate of the 
Prym variety of $\Sigma$ with respect to $\sigma$ for all $x \in T^2 $. Since the Prym variety is complex $1$ dimensional and $\rho^*\mathcal L_x = \mathcal L_x \jj$, the image of the map
$\Psi $
is real $1-$dimensional and has a $1-$dimensional kernel. By the same arguments as in Proposition \ref{g=1} we obtain that $f$ is equivariant.
\end{proof}


\end{document}